\documentclass[a4paper,11pt,english]{article}
\usepackage{amsmath}
\usepackage{amsfonts}
\usepackage{amssymb}
\usepackage{amsthm}
\usepackage[all]{xy}
\usepackage{geometry}
\usepackage{babel}

\newtheorem {thm}{Theorem}
\newtheorem* {thm*}{Theorem}

\newtheorem {quest}[thm]{Question}
\newtheorem {lem}[thm]{Lemma}

\theoremstyle{definition}

\DeclareMathOperator{\End}{End}

\DeclareMathOperator{\Hom}{Hom}

\title{On the problem of detecting linear dependence for products of abelian varieties and tori}
\author{Antonella Perucca}
\date{}
\begin{document}
\maketitle

\begin{abstract}
Let $G$ be the product of an abelian variety and a torus defined over a number field $K$.
Let $R$ be a point in $G(K)$ and let $\Lambda$ be a finitely generated subgroup of $G(K)$. Suppose that for all but finitely many primes $\mathfrak p$ of $K$ the point $(R \bmod \mathfrak{p})$ belongs to $(\Lambda \bmod \mathfrak{p})$. Does it follow that $R$ belongs to $\Lambda$?
We answer this question affirmatively in three cases: if $\Lambda$ is cyclic; if $\Lambda$ is a free left $\End_K G$-submodule of $G(K)$; if $\Lambda$ has a set of generators (as a group) which is a basis of a free left $\End_K G$-submodule of $G(K)$.
In general we prove that there exists an integer $m$ (depending only on $G$, $K$ and the rank of $\Lambda$) such that $mR$ belongs to the left $\End_K G$-submodule of $G(K)$ generated by $\Lambda$.
\end{abstract}

\section{Introduction}\label{introlindep}

The problem of detecting linear dependence investigates whether the property for a rational point to belong to a subgroup obeys a local-global principle.

\begin{quest}
Let $G$ be the product of an abelian variety and a torus defined over a number field $K$.
Let $R$ be a point in $G(K)$ and let $\Lambda$ be a finitely generated subgroup of $G(K)$. Suppose that for all but finitely many primes $\mathfrak p$ of $K$ the point $(R \bmod \mathfrak{p})$ belongs to $(\Lambda \bmod \mathfrak{p})$. Does $R$ belong to $\Lambda$?
\end{quest}

We answer this question affirmatively in three cases: if $\Lambda$ is cyclic; if $\Lambda$ is a free left $\End_K G$-submodule of $G(K)$; if $\Lambda$ has a set of generators (as a group) which is a basis of a free left $\End_K G$-submodule of $G(K)$.
In general we prove that there exists an integer $m$ (depending only on $G$, $K$ and the rank of $\Lambda$) such that $mR$ belongs to the left $\End_K G$-submodule of $G(K)$ generated by $\Lambda$.\newline

The problem of detecting linear dependence was first formulated by Gajda in 2002 in a letter to Ken Ribet.
Papers and preprints concerning this problem are: \cite{Schinzelpower}, \cite{Kharegalois}, \cite{Weston}, \cite{Kowalskikummer}, \cite{BGKdetecting}, \cite{Banaszaklindep}, \cite{Baranczuk08}.

We now give the state of the art of the problem of detecting linear dependence for abelian varieties.

\begin{itemize}
\item The strongest result is by Weston in \cite{Weston}:
if the abelian variety has commutative endomorphism ring then there exists a $K$-rational torsion point $T$ such that $R+T$ belongs to $\Lambda$. Since the torsion of the Mordell-Weil group is finite, Weston basically solved the problem for abelian varieties with commutative endomorphism ring. It is not known how to get rid of this torsion point. %We remove it in a special case, see Theorem~\ref{westonkowalski}.

\item If the endomorphism ring of the abelian variety is not commutative, we are able to prove the following:
there exists a non-zero integer $m$ (depending only on $G$ and $K$) such that $mR$ belongs to the $\End_K G$-submodule of $G(K)$ generated by $\Lambda$, see Theorem~\ref{lindepthm}.

\item We solve the problem of detecting linear dependence in the case where $\Lambda$ is a free $\End_K G$-submodule of $G(K)$ or if $\Lambda$ has a set of generators (as a group) which is a basis of a free $\End_K G$-submodule of $G(K)$.
With an extra assumption on the point $R$ (that $R$ generates a free left $\End_K G$-submodule of $G(K)$), these two results are respectively proven by Gajda and G\'ornisiewicz in \cite[Theorem B]{GajdaGornisiewicz} and by Banaszak in \cite[Theorem 1.1]{Banaszaklindep}.
We remove the assumption on the point $R$ in Theorem~\ref{lindepthm} and in Theorem~\ref{banstronger} respectively.

\item If $\Lambda$ is cyclic, we solve the problem of detecting linear dependence. This result was known only for elliptic curves, see \cite[Theorem 3.3]{Kowalskikummer} by Kowalski.

\item Gajda and G\'ornisiewicz in \cite{GajdaGornisiewicz} use the theory of integrally-semisimple Galois modules to study the problem of detecting linear dependence. This theory was completely developed by Larsen in \cite{Larsenwhitehead}. Gajda and G\'ornisiewicz prove the following result (\cite[Theorem A]{GajdaGornisiewicz}):

Let $\ell$ be a prime such that $T_\ell (G)$ is integrally semisimple, let $\hat{\Lambda}$ be a free $\End_K G \otimes \mathbb Z_\ell$-submodule of $G(K)\otimes \mathbb Z_\ell$ and let $\hat{R}$ in $G(K)\otimes \mathbb Z_\ell$ generate a free $\End_K G \otimes \mathbb Z_\ell$-submodule of $G(K)\otimes \mathbb Z_\ell$. Then $\hat{R}$ belongs to $\hat{\Lambda}$ if and only if for all but finitely many primes $\mathfrak p$ of $K$ $(\hat{R} \bmod \mathfrak p)$ belongs to $(\hat{\Lambda} \bmod \mathfrak p)$. If $\End_K G \otimes \mathbb Z_\ell$ is a maximal order in $\End_K G \otimes \mathbb Q_\ell$, the condition on $\hat{\Lambda}$ can be replaced by the following: $\hat{\Lambda}$ is torsion-free over $\End_K G \otimes \mathbb Z_\ell$.
%\item If $G$ is an abelian variety with endomorphism ring $\mathbb Z$ (or the multiplicative group), Bara\'nczuk proves the following in \cite[Theorem 3.1]{Baranczuk08}: Let $R_1,\ldots, R_n$ be points in $G(K)$ which generate a free $\End_K G$-submodule. Then there exists an index $i$ and a non-zero integer $m$ such that $mR_i$ belongs to $\Lambda$ if and only if for all but finitely many primes $\mathfrak p$ of $K$ there exists an index $i_{\mathfrak p}$ such that $(R_{i_\mathfrak p} \bmod \mathfrak p)$ belongs to $(\Lambda \bmod \mathfrak p)$. We claim that this result can be straight-forwardly generalized to simple abelian varieties by replacing $m$ with an isogeny and by assuming (as in this special case) that there exists a non-zero integer $a$ such that $a\Lambda$ is a free $\End_K G$-submodule of $G(K)$.
\end{itemize}

Now we list further results on the problem of detecting linear dependence for commutative algebraic groups.
Schinzel in \cite[Theorem 2]{Schinzelpower} solved the problem of detecting linear dependence for the multiplicative group.
A generalization of Schinzel's result (Lemma~\ref{banladic} for the multiplicative group with no conditions on $\Lambda$) was proven by Khare in \cite[Proposition 3]{Kharegalois}.% by applying the method by Corrales-Rodrig\'a\~nez and Schoof (see \cite{CorralesSchoof}).

%Our results on the problem of detecting linear dependence (Theorems~\ref{lindepthm}, \ref{banstronger} and \ref{westonkowalski}) are proven in the generality of the product of an abelian variety and a torus.

Kowalski in \cite{Kowalskikummer} studied the problem of detecting linear dependence in the case where $\Lambda$ is cyclic. In particular he investigates for which commutative algebraic groups one can solve the problem of detecting linear dependence. This is not possible whenever the additive group is embedded into $G$, see \cite[Proposition 3.2]{Kowalskikummer}. Hence it is left to treat the case of semi-abelian varieties.

Finally, a variant of the problem of detecting linear dependence was considered by Bara\'nczuk in \cite{Baranczuk08} for the multiplicative group and abelian varieties with endomorphism ring $\mathbb Z$.

\section{Preliminaries}\label{sectionpreliminaries}

Let $G$ be the product of an abelian variety and a torus defined over a number field $K$. Let $R$ be a $K$-rational point on~$G$ and call $G_R$ the smallest algebraic $K$-subgroup of $G$ containing $R$. 
Write $G_R^0$ for the connected component of the identity of $G_R$ and write $n_R$ for the number of connected components of $G_R$.
By \cite[Proposition 5]{Peruccaorder}, $G^0_R$ is the product of an abelian variety and a torus defined over $K$.

 We say that $R$ is \textit{independent} if $R$ is non-zero and $G_R=G$.
The point $R$ is independent in $G$ if and only if $R$ is independent in $G\times_K \bar{K}$. Furthermore, $R$ is independent in $G$ if and only if $R$ is non-zero and the left $\End_K G$-submodule of $G(K)$ generated by $R$ is free. See \cite[Section 2]{Peruccaorder}.

\begin{lem}\label{dimGdR}
Let $R$ be a $K$-rational point on~$G$ and let $d$ be a non-zero integer. We have $G^0_{dR}= G^0_R$. In particular, the dimension of  $G_{dR}$ equals the dimension of $G_{R}$ and $G_{n_RR}=G_R^0$.
\end{lem}
\textit{Proof.}
Since $G_R$ contains $dR$ we have $G_{dR}\subseteq G_R$ and so  $G^0_{dR}\subseteq G^0_R$. Hence it suffices to prove that $G^0_{dR}$ and $G^0_{R}$ have the same dimension. Clearly the dimension of $G^0_{dR}$ is less than or equal to the dimension of $G^0_R$. To prove the other inequality it suffices to show that the multiplication by $[d]$ maps $G_R$ into $G_{dR}$. This is true because $[d]^{-1}G_{dR}$ contains $R$.
\hfill$\square$\newline

Call $W$ the connected component of $G_R$ containing $R$ and let $X$ be a torsion point in $G_R(\bar{K})$ such that $W=X+G^0_R$ (see \cite[Lemma 1]{Peruccaorder}). Clearly $n_RX$ is the least positive multiple of $X$ belonging to $G^0_R$ and the connected components of $G_R$ are of the form $aX+G^0_R$ for $0\leq a<n_R$. 
We can write $R=X+Z$ where $Z$ is in $G^0_R(\bar{K})$. Since $R$ and $Z$ have a common multiple, from Lemma~\ref{dimGdR} it follows that $Z$ is independent in $G^0_R$. 

\begin{lem}\label{cosetXmultiple}
Let $L$ be a finite extension of $K$ where $X$ is defined.
Then for all but finitely many primes $\mathfrak q$ of $L$ the point $(n_R X \bmod\mathfrak q)$ is the least multiple of $(X \bmod\mathfrak q)$ belonging to $(G_R^0 \bmod\mathfrak q)$.
\end{lem}
\begin{proof}
Call $x$ the order of $X$. We may assume that the points in $G_R[x]$ are defined over $L$.
Suppose that $d$ is an integer not divisible by $n_R$ such that for infinitely many primes $\mathfrak q$ of $L$ the point $(dX \bmod \mathfrak p)$ belongs to $(G_R^0 \bmod\mathfrak q)$.
Up to excluding finitely many primes $\mathfrak q$, we may assume that the reduction modulo $\mathfrak q$ maps injectively $G_R[x]$ to $(G_R \bmod \mathfrak q)[x]$. By \cite[Lemma 4.4]{Kowalskikummer} we may also assume that  the reduction modulo $\mathfrak q$ maps surjectively $G^0_R[x]$ onto $(G^0_R \bmod \mathfrak q)[x]$.
Then for infinitely many primes $\mathfrak q$ the point $(dX \bmod\mathfrak q)$ belongs to the reduction modulo $\mathfrak q$ of the finite group $G^0_R[x]$. We deduce that $dR$ belongs to $G^0_R[x]$. We have a contradiction since $n_RX$ is the least positive multiple of $X$ which belongs to $G_R^0$.
\end{proof}

\begin{lem}\label{zerohom}
Let $A$ and $T$ be respectively an abelian variety and a torus defined over a number field $K$.
Then $\Hom_{\bar{K}}(A,T)=\{0\}$ and $\Hom_{\bar{K}}(T,A)=\{0\}$.
\end{lem}
\begin{proof}
Since $A$ is a complete variety and $T$ is affine, there are no non-trivial morphisms from $A$ to $T$.
To prove the other equality, suppose that $\phi$ is a non-zero morphism from $\mathbb G_m$ to $A$. Then $\phi(\mathbb G_m)$ is connected and has dimension $1$. We deduce that $\phi$ is an isogeny from $\mathbb G_m$ to an elliptic curve. This is impossible by the Hurwitz formula (\cite[Chapter II Theorem 5.9]{Silvermanecbook}).
\end{proof}

\begin{lem}
Let $R$ be a commutative ring with $1$. Let $F$ be a free $R$-module. Suppose that $s$ is an $R$-endomorphism of $F$ sending every element to a multiple of itself. Then $s$ is a scalar.
\end{lem}
\textit{Proof.}
It suffices to prove the statement if $F$ has rank $2$. Let $e_1,e_2$ be a basis of $F$. Then $s(e_1)=\lambda_1 e_1$ and $s(e_2)=\lambda_2 e_2$ and $s(e_1 + e_2) = \mu(e_1 + e_2)$ for some $\lambda_1$, $\lambda_2$, $\mu$ in $R$. We deduce that $\lambda_1 = \lambda_2 = \mu$ therefore $s$ is the multiplication by $\mu$ on $F$.
\hfill $\square$
\newline

The following lemma in the case of abelian varieties is proven in \cite[Step 2 of the proof of Theorem 1.1]{Banaszaklindep}.

\begin{lem}\label{scalarontorsion}
Let $G$ be the product of an abelian variety and a torus defined over a number field $K$. Let $\alpha$ be a $\bar{K}$-endomorphism of $G$.
Suppose that there exists a prime number $\ell$ such that the following holds: for every $n>0$ and for every torsion point $T$ of $G$ of order $\ell^n$ the point $\alpha(T)$ is a multiple of $T$.
Then $\alpha$ is a scalar.
\end{lem}
\begin{proof}

Apply the previous lemma to $R=\mathbb Z/\ell^n \mathbb Z$, $F=
G[\ell^n]$ and taking for $s$ the image of $\alpha$ in $\End_{\mathbb Z} G[\ell^n]$. We deduce that $\alpha$ acts as a scalar on $G[\ell^n]$. 
So for every $n>0$ there exists an integer $c_n$ such that $\alpha$ acts as the multiplication by $c_n(\bmod \ell^n)$ on $G[\ell^n]$.
Since $\alpha$ commutes with the multiplication by $\ell$ we deduce that $c_{n+1}(\bmod \ell^n)\equiv c_{n}(\bmod \ell^n)$ for every $n$.
This means that there exists $c$ in $\mathbb Z_\ell$ such that $c(\bmod \ell^n)\equiv c_{n}(\bmod \ell^n)$ for every $n$ and that $\alpha$ acts on $T_\ell G$ as the multiplication by $c$.

Write $G=A\times T$ where $A$ is an abelian variety and $T$ is a torus. By Lemma~\ref{zerohom}, $\alpha$ is the product $\alpha_A\times \alpha_T$ of an endomorphism of $A$ and an endomorphism of $T$. It suffices to prove the following: if $A$ (respectively $T$) is non-zero then $c$ is an integer and $\alpha_A$ (respectively $\alpha_T$) is the multiplication by $c$.

Suppose that $A$ is non-zero. We know that $\alpha_A$ acts on $T_\ell A$ as the multiplication by $c$. By \cite[Theorem 3 p.176]{Mumfordbook}, $\alpha_A$ is the multiplication by an integer. Consequently, $c$ is an integer and $\alpha_A$ is the multiplication by $c$.

Suppose that $T$ is non-zero. We reduce at once to the case where $T=\mathbb G_m^h$ for some $h\geq 1$. 
The endomorphism ring of $\mathbb G_m$ is $\mathbb Z$ hence we can identify the endomorphism ring of $T$ with the ring of $h\times h$-matrices with integer coefficients. Since $\alpha_{T}$ acts on $T_\ell T$ as the multiplication by $c$, we deduce that $\alpha_{T}$ is a scalar matrix. Hence $c$ is an integer and $\alpha_{T}$ is the multiplication by $c$.
\end{proof}

\section{An application of the results on the support problem} \label{sectionlindep}

\begin{thm}\label{lindepthm}
Let $G$ be the product of an abelian variety and a torus defined over a number field $K$.
Let $R$ be a $K$-rational point on $G$ and let $\Lambda$ be a finitely generated subgroup of $G(K)$. Suppose that for all but finitely many primes $\mathfrak p$ of $K$ the point $(R \bmod\mathfrak{p})$ belongs to $(\Lambda \bmod\mathfrak{p})$. Then there exists a non-zero integer $m$ (depending only on $G$, $K$ and the rank of $\Lambda$) such that $mR$ belongs to the left $\End_K G$-submodule of $G(K)$ generated by $\Lambda$.
Furthermore, if $\Lambda$ is a free left $\End_K G$-submodule of $G(K)$ then $R$ belongs to $\Lambda$.
\end{thm}

Remark that if $G$ is an abelian variety, the integer $m$ in Theorem~\ref{lindepthm} depends only on $G$ and $K$ since the rank of $\Lambda$ is bounded by the rank of the Mordell-Weil group.

\begin{lem}\label{lindep}
Let $G$ be the product of an abelian variety and a torus defined over a number field $K$.
Let $R$ be a $K$-rational point on $G$ and let $\Lambda$ be a finitely generated subgroup of $G(K)$. Fix a rational prime $\ell$.
Suppose that for all but finitely many primes $\mathfrak p$ of $K$ there exists an integer $c_\mathfrak p$ coprime to $\ell$ such that $(c_\mathfrak pR \bmod \mathfrak p)$ belongs to $(\Lambda\bmod \mathfrak p)$.
Then there exists a non-zero integer $c$ such that $cR$ belongs to the left $\End_K G$-submodule of $G(K)$ generated by $\Lambda$.
One can take $c$ such that $v_\ell(c)\leq v_\ell(m)$ where $m$ is a non-zero integer depending only on $G$, $K$ and the rank of $\Lambda$ (hence not depending on $\ell$).
If $\Lambda$ is a free left $\End_K G$-submodule of $G(K)$, one can take $m=1$.
\end{lem}
\textit{Proof.} 
Let $P_1,\ldots,P_s$ generate $\Lambda$ as a $\mathbb Z$-module. Consider $G^s$ and its $K$-rational points $P=(P_1,\dotsc,P_s)$ and $Q=(R,0,\dotsc,0)$. We can apply \cite[Main Theorem]{Peruccatwo} to the points $P$ and $Q$. Then there exist a $K$-endomorphism $\phi$ of $G^s$ and a non-zero integer $c$ such that $\phi(P)=cQ$.
By \cite[Proposition 9]{Peruccatwo} one can take $c$ such that $v_\ell(c)\leq v_\ell(m)$ where $m$ depends only on $G^s$ and $K$.
In particular $cR$ belongs to $\End_K G\cdot \Lambda$.
Since $s$ depends only on $G$, $K$ and the rank of $\Lambda$, the first assertion is proven.
For the second assertion, let $P_1,\ldots,P_s$ be a basis of $\Lambda$ as a left $\End_K G$-module. Since $P$ is independent, by \cite[Proposition 8]{Peruccatwo} one can take $c$ coprime to $\ell$. Consequently, one can take $m=1$.
\hfill $\square$
\newline

\textit{Proof of Theorem~\ref{lindepthm}.}
We apply Lemma~\ref{lindep} to every rational prime $\ell$. 
Then for every $\ell$ there exists an integer $c_\ell$ such that $c_\ell R$ belongs to $\End_K G\cdot \Lambda$ and $v_\ell(c_\ell)\leq v_\ell(m)$, where
$m$ is a non-zero integer depending only on $G$, $K$ and the rank of $\Lambda$.
Since $m$ is in the ideal of $\mathbb Z$ generated by the $c_\ell$'s, we deduce that $mR$ belongs to $\End_K G\cdot \Lambda$.
If $\Lambda$ is a free left $\End_K G$-submodule of $G(K)$, one can take $m=1$ in Lemma~\ref{lindep} hence $R$ belongs to $\Lambda$.
\hfill $\square$

\section{A refinement of a result by Banaszak}\label{sectionbanaszak}

In this section we extend the result by Banaszak on the problem of detecting linear dependence (\cite[Theorem 1.1]{Banaszaklindep}) from abelian varieties to products of abelian varieties and tori. Furthermore, by adapting Banaszak's proof we are able to remove his assumption on the point $R$ (that $R$ generates a free left $\End_K G$-submodule of $G(K)$).

\begin{thm}\label{banstronger}
Let $G$ be the product of an abelian variety and a torus defined over a number field $K$. 
Let $\Lambda$ be a finitely generated subgroup of $G(K)$ such that it has a set of generators (as a group ) which is a basis of a free left $\End_K G$-submodule of $G(K)$. Let $R$ be a point of $G(K)$. Suppose that for all but finitely many primes $\mathfrak p$ of $K$ the point $(R \bmod \mathfrak p)$ belongs to $(\Lambda \bmod \mathfrak p)$. Then $R$ belongs to $\Lambda$.
\end{thm}

\begin{lem}\label{banladic}
Let $G$ be the product of an abelian variety and a torus defined over a number field $K$.
Let $\Lambda$ be a finitely generated subgroup of $G(K)$ such that it has a set of generators (as a group) which is a basis of a free left $\End_K G$-submodule of $G(K)$. Let $R$ be a point in $G(K)$.
 Fix a prime number $\ell$.
Suppose that for all but finitely many primes $\mathfrak p$ of $K$ there exists an integer $c_{\mathfrak p}$ coprime to $\ell$ such that 
the point $(c_{\mathfrak p}R \bmod \mathfrak p)$ belongs to $(\Lambda \bmod \mathfrak p)$. Then there exists an integer $c$ coprime to $\ell$ such that $cR$ belongs to $\Lambda$.
\end{lem}
\begin{proof}
By Lemma~\ref{lindep} applied to $\End_K G \cdot\Lambda$, there exists an integer $c$ coprime to $\ell$ such that $cR$ belongs to $\End_K G\cdot \Lambda$. Let $\{P_1,..P_n\}$ be a set of generators for $\Lambda$ which is a basis for $\End_K G \cdot\Lambda$. We can write
$$cR=\sum_{i=1}^n \phi_i P_i$$
 for some $\phi_i$ in $\End_K G$. Without loss of generality it suffices to prove that $\phi_1$ is the multiplication by an integer.

Suppose that $\phi_1$ is not the multiplication by an integer and apply Lemma~\ref{scalarontorsion} to $\phi_1$. Then there exists a point $T$ in $G[\ell^\infty]$ such that $\phi_1(T)$ is not a multiple of $T$.
Let $L$ be a finite extension of $K$ where $T$ is defined.
The point $(P_1-T,P_2,\ldots,P_n)$ is independent in $G^n$ hence by \cite[Proposition 12]{Peruccaorder} there are infinitely many primes $\mathfrak q$ of $L$ such that the following holds:
$(P_i \bmod {\mathfrak q})$ has order coprime to $\ell$ for every $i\neq 1$ and $(P_1-T \bmod {\mathfrak q})$ has order coprime to $\ell$.
By discarding finitely many primes $\mathfrak q$, we may assume the following: the order of $(T \bmod {\mathfrak q})$ equals the order of $T$; the point $(\phi_1(T) \bmod {\mathfrak q})$ is not a multiple of $(T \bmod {\mathfrak q})$ and in particular it is non-zero; $(c_{\mathfrak q}R \bmod {\mathfrak q})$ belongs to $(\Lambda \bmod {\mathfrak q})$ for some integer $c_{\mathfrak q}$ coprime to $\ell$. 

Fix $\mathfrak q$ as above. We know that there exists an integer $m$ coprime to $\ell$ such that $(mP_i \bmod {\mathfrak q})=0$ for every $i\neq 1$ and $(m(P_1-T)\bmod {\mathfrak q})=0.$
Then we have: $$(mc_{\mathfrak q}cR \bmod {\mathfrak q})=(mc_{\mathfrak q} \phi_1(P_1) \bmod {\mathfrak q})=(mc_{\mathfrak q} \phi_1(T) \bmod {\mathfrak q}).$$
Since $v_\ell(mc_{\mathfrak q})=0$, we deduce that the point $(mc_{\mathfrak q}cR \bmod {\mathfrak q})$ has order a power of $\ell$ and it is not a multiple of $(T \bmod {\mathfrak q})$.
Then $(mc_{\mathfrak q}cR \bmod {\mathfrak q})$ does not belong to $\sum_{i=1}^n \mathbb Z (P_i \bmod {\mathfrak q})$. 
Consequently, $(c_{\mathfrak q}R \bmod {\mathfrak q})$ does not belong to $(\Lambda \bmod {\mathfrak q})$ and we have a contradiction.
\end{proof}

\begin{proof}[Proof of Theorem~\ref{banstronger}]
We can apply Lemma~\ref{banladic} to every rational prime $\ell$. Then for every $\ell$ there exists an integer $c_\ell$ coprime to $\ell$ such that $c_\ell R$ belongs to $\Lambda$. Since $1$ is contained in the ideal of $\mathbb Z$ generated by the $c_\ell$'s, we deduce that $R$ belongs to  $\Lambda$.
\end{proof}

\section{On a problem by Kowalski}\label{sectionkowalski}

\begin{thm}\label{westonkowalski}
Let $G$ be the product of an abelian variety and a torus defined over a number field $K$. Let $\Lambda$ be a cyclic subgroup of $G(K)$. Let $R$ be a $K$-rational point on $G$. Suppose that for all but finitely many primes $\mathfrak p$ of $K$ the point $(R \bmod\mathfrak p)$ belongs to $(\Lambda \bmod\mathfrak p)$. Then $R$ belongs to $\Lambda$. 
\end{thm}

\begin{lem}\label{westonkowalskilem}
Let $G$ be the product of an abelian variety and a torus defined over a number field $K$. Let $\Lambda$ be a cyclic infinite subgroup of $G(K)$. Let $T$ be a $K$-rational torsion point on $G$. Suppose that for all but finitely many primes $\mathfrak p$ of $K$ the point $(T \bmod\mathfrak p)$ belongs to $(\Lambda \bmod\mathfrak p)$. Then $T$ is zero.
\end{lem}
\begin{proof}
Suppose that $T$ is non-zero. Then $T$ can be uniquely written as a sum of torsion points whose orders are prime powers. These torsion points are multiples of $T$. Consequently, we reduce at once to the case where the order of $T$ is the power of a prime number $\ell$.

Let $\Lambda=\mathbb Z P$ for a point $P$ of infinite order.
The algebraic subgroup $G_P$ of $G$ generated by $P$ has dimension at least $1$. In section~\ref{sectionpreliminaries} we saw the following: $P=X+Z$ for some point $Z$ in $G_P^0(\bar{K})$ and some torsion point $X$ in $G_P(\bar{K})$; the point $Z$ is independent in $G_P^0$; $n_PX$ is the least multiple of $X$ which belongs to $G_P^0$; $G_P^0$ is the product of an abelian variety and a torus defined over $K$. %(and it is non-zero since $P$ has infinite order).

Let $c$ be the $\ell$-adic valuation of the order of $X$.
Let $L$ be a finite extension of $K$ where $X$, $Z$, $G[\ell^{2c}]$ are defined and such that $n_PX$ has $n_P$-roots in $G_P^0(L)$.
Notice that for all but finitely many primes $\mathfrak q$ of $L$ the point $(T \bmod\mathfrak q)$ belongs to $(\mathbb ZP \bmod\mathfrak q)$.

By \cite[Proposition 12]{Peruccaorder}, there exist infinitely many primes $\mathfrak q$ of $L$ such that the order of $(Z \bmod\mathfrak q)$ is coprime to $\ell$.
Then for infinitely many primes $\mathfrak q$ the point $(T \bmod\mathfrak q)$ lies in the finite group generated by $(X \bmod\mathfrak q)$.
We deduce that $T=aX$ for some non-zero integer $a$.

Let $T_0$ be a point in $G_P^0$ of order $\ell^{2c}$.
By \cite[Proposition 11]{Peruccaorder}, there exist infinitely many primes $\mathfrak q$ of $L$ such that the order of $(Z-T_0 \bmod\mathfrak q)$ is coprime to $\ell$.
We deduce that for infinitely many primes $\mathfrak q$ the point $(T \bmod\mathfrak q)$ lies in the finite group generated by $(T_0 + X \bmod\mathfrak q)$.
Then $T=b(T_0+X)$ for some non-zero integer $b$.

Since $aX=b(T_0+X)$ and because the order of $T_0$ is $\ell^{2c}$ we deduce that $v_\ell(b)\geq c$. Consequently, $T$ is the sum of $bT_0$ and a torsion point of order coprime to $\ell$. Then $T$ is a multiple of $T_0$ and in particular it belongs to $G_P^0$.

Let $T_1$ be a point in $G_P^0(L)$ such that $n_P T_1=-n_P X$.
By \cite[Proposition 11]{Peruccaorder}, there exist infinitely many primes $\mathfrak q$ of $L$ such that the order of $(Z-T_1 \bmod\mathfrak q)$ is coprime to $\ell$.
Up to discarding finitely many primes $\mathfrak q$, we may assume that $(T \bmod\mathfrak q)$ belongs to $(\mathbb ZP \bmod\mathfrak q)$ and that the order of $(T \bmod\mathfrak q)$ equals the order of $T$.
Up to discarding finitely many primes $\mathfrak q$, by Lemma~\ref{cosetXmultiple} we may assume that $(n_P X \bmod\mathfrak q)$ is the least multiple of $(X \bmod\mathfrak q)$ belonging to $(G_P^0 \bmod\mathfrak q)$.
Consequently, the intersection of $(G_P^0 \bmod\mathfrak q)$ and $(\mathbb ZP \bmod\mathfrak q)$ is $(\mathbb Zn_PP \bmod\mathfrak q)$. Then $(T \bmod\mathfrak q)$ belongs to $(\mathbb Zn_PP \bmod\mathfrak q)$.

Fix a prime $\mathfrak q$ as above and call $r$ the order of $(Z-T_1 \bmod\mathfrak q)$. We have
$$(r n_P P \bmod\mathfrak q)=(r n_P  Z+ r n_P X \bmod\mathfrak q)=(r n_PT_1+ r n_PX \bmod\mathfrak q)=(0 \bmod\mathfrak q).$$
Since $r$ is coprime to $\ell$, it follows that $(\mathbb Zn_PP \bmod\mathfrak q)$ has no $\ell$-torsion and in particular it does not contain  $(T \bmod\mathfrak q)$. We have a contradiction.
\end{proof}

\begin{proof}[Proof of Theorem~\ref{westonkowalski}]
If $\Lambda$ is finite then there exists an element $P'$ in $\Lambda$ such that for infinitely many primes $\mathfrak p$ of $K$ it is $(R \bmod\mathfrak p)=(P' \bmod\mathfrak p)$. Hence $R=P'$ and the statement is proven.
We may then assume that $\Lambda=\mathbb Z P$ for a point $P$ of infinite order.

We first prove that the statement holds in the case where the algebraic group $G_P$ generated by $P$ is connected.
In this case, $G_P$ is the product of an abelian variety and a torus (\cite[Proposition 5]{Peruccaorder}). By Lemma~\cite[Lemma 4.2]{Kowalskikummer}, we may assume that $G_P=G$. So we may assume that $P$ is independent in $G$ and we conclude by applying Theorem~\ref{banstronger}.

In general, call $n_P$ the number of connected components of $G_P$. Notice that the points $n_P P$ and $n_P R$ still satisfy the hypotheses of the theorem and that $G_{n_P P}$ is connected by Lemma~\ref{dimGdR}. 
Therefore we know (by the special case above) that $n_P R=g n_P P$ for some integer $g$.
Since $R$ and $P$ are rational points, we deduce that $R=gP+T$ for some rational torsion point $T$. 
Since $R+T$ belongs to $\Lambda$, for all but finitely many primes $\mathfrak p$ of $K$ the point $(T \bmod\mathfrak p)$ belongs to $(\Lambda \bmod\mathfrak p)$.
By applying Lemma~\ref{westonkowalskilem} we deduce that $T=0$ hence $R$ belongs to $\Lambda$.
\end{proof}

\section*{Acknowledgements}
I thank Emmanuel Kowalski and Ren\'e Schoof for helpful discussions.

%\bibliographystyle{amsalpha}
%\bibliography{bibliography.bib}

\end{document}